\documentclass[a4paper]{article}

\usepackage[utf8]{inputenc}
\usepackage[T1]{fontenc}

\usepackage{amsmath, amsthm, amssymb}
\usepackage{enumerate}
\usepackage{url}

\usepackage{graphics}
\usepackage{subfig}

\theoremstyle{plain}
\newtheorem{thm}{Theorem}
\newtheorem{prop}[thm]{Proposition}
\newtheorem{propdefn}[thm]{Proposition-definition}
\newtheorem{lem}[thm]{Lemma}
\newtheorem{cor}[thm]{Corollary}

\theoremstyle{definition}
\newtheorem{defn}{Definition}
\newtheorem{exm}{Example}

\theoremstyle{remark}

\newcommand{\eps}{\varepsilon}
\newcommand{\trifun}[4]{{{#1_{#2}}^{#3}}_{#4}}
\newcommand{\setsuch}[2]{\left\{ #1 \; \middle| \; #2 \right\}}
\newcommand{\maxof}{\max\,}
\newcommand{\minof}{\min\,}
\DeclareMathOperator{\tr}{tr}
\DeclareMathOperator{\interior}{int}
\DeclareMathOperator{\sgn}{sgn}

\begin{document}

\title{Harmonic functions on the Sierpinski triangle}
\author{Ilia Smilga}

\maketitle

\begin{abstract}
In this paper, we give a few results on the local behavior of harmonic functions on the Sierpinski triangle - more precisely, of their restriction to a side of the triangle. First we present a general formula that gives the Hölder exponent of such a function in a given point. From this formula, we deduce an explicit algorithm to calculate this exponent in any rational point, and the fact that the derivative of such a function is always equal to $0$, $\infty$ or undefined.
\end{abstract}

\section{Introduction}
\label{sec:intro}

\subsection{Notations and conventions}
\label{sec:notations}
In all this text, $\omega$ stands for the cubic root of $-1$ with positive imaginary part.

For $z \in \mathbb{C}$, we define $h_z: \mathbb{C} \to \mathbb{C}$ to be the homothety of center $z$ and ratio $\frac{1}{2}$: $h_z(w) = \frac{w+z}{2}$.

$\mathbb{N}$ stands for the set of all nonnegative integers.

When talking about a positive quantity, we shall say that it is \emph{positively bounded} if it is bounded from above by a finite constant and from below by a positive constant. This is equivalent to saying that its logarithm is bounded.

When we say that some derivative is well-defined, we shall mean that the rate of change has a limit in $\bar{\mathbb{R}}$, i. e. we shall also implicitly admit infinitive derivatives.

We will often need to make estimations up to a positively bounded multiplicative constant. As all norms on a finite-dimensional vector space are equivalent, in such a context, we have no need to distinguish them. So we shall usually simply write ``$\|\bullet\|$'', implying that the statement we are making is true for any norm. However, to perform some calculations, we will need to use some specific norms: we shall then use an index to specify which norm we are talking about.

We will also often need to manipulate products of the form $F_{a_1} \ldots F_{a_n}$, where $F$ is some kind of operator that can take one of two values $F_0$ and $F_1$ (for example $h$, $M$, $\vec{M}$, $\tilde{M}$ and so on). We shall then simply write $F_{a_1 \ldots a_n}$, or even $F_w$ (if we have already agreed that $w = a_1 \ldots a_n$), as shorthand for such products.

Everything written in positional notation with a radix point is assumed to be in binary (for example $\frac{1}{2} = 0.1$), unless the context clearly shows otherwise (for example $\pi \approx 3.14\ldots$).

As everyone knows, dyadic rational numbers have two binary expansions. Being given $s \in [0, 1)$ (resp. $s \in (0, 1]$), we call its \emph{upper} (resp. \emph{lower}) \emph{binary expansion} the unique binary expansion of $s$ that does not end with an infinite sequence of $1$'s (resp. of $0$'s). For the problems that we will study, both expansions will play perfectly symmetrical roles; hence we will often simply say ``binary expansion'' without specifying, and we shall mean ``upper or lower binary expansion''.

\subsection{Statement of the problem}
\label{sec:problem}

Let $u$ be the unique continuous fonction from $[0,1]$ to $\mathcal{R}^3$ that satisfies:

\begin{equation}\label{eq:recrelu}
\begin{cases}
u(0) = \begin{pmatrix}1\\0\\0\end{pmatrix},\quad   u(1) = \begin{pmatrix}0\\1\\0\end{pmatrix} \\
u \circ h_0 = M_0 \circ u \\
u \circ h_1 = M_1 \circ u,\\
\end{cases}
\end{equation}
where we set
\[M_0 := \frac{1}{5}
\begin{pmatrix}
5 & 2 & 2\\
0 & 2 & 1\\
0 & 1 & 2\\
\end{pmatrix},\quad
M_1 := \frac{1}{5}
\begin{pmatrix}
2 & 0 & 1\\
2 & 5 & 2\\
1 & 0 & 2\\
\end{pmatrix}\]
(see also Figure \ref{fig:ut}). We shall justify later the existence and the unicity of this function.

In \cite{kirillov}, Kirillov formulates the question (called "Problem 2", near the end of Section 3.2) :
\begin{quote}
\emph{Compute explicitly the derivative $u'(t)$ whenever it is possible (e.g. at all rational points).}
\end{quote}
The purpose of our paper is to answer this question. Note however that our notations are slightly different from Kirillov's: what he calls $u$ corresponds in fact to a projection of the function that we have decided to call $u$. Thus we shall actually deal with a slightly more general question.

\subsection{Plan of the paper}
\label{sec:plan}

We start, in section \ref{sec:prelim}, by explaining the interest of the function $u$. While mostly following Kirillov's book, we construct harmonic functions on the Sierpinski triangle, and we show that $u$ appears as the restriction of some kind of "universal" harmonic function to a side of the triangle. The rest of the article is logically independent from this section, except for a few notations and definitions. (For a more thorough introduction to harmonic functions on the Sierpinski triangle and on other fractals, see also Strichartz's book \cite{strichartz}.)

In section \ref{sec:general}, we establish a few general results about that function.

In the following, we study the behaviour of $u(t) - u(s)$ as $t \to s$. At first (section \ref{sec:direction}), we completely describe the behaviour of the direction of this vector. This gives the general shape of the curve $u(t)$, and allows us, a few pages further, to link the local properties of the functions that Kirillov studies with those of $u$.

In section \ref{sec:norm}, we study the norm of this vector. The key result of this paper is Proposition \ref{Hölder exponent general formula}, that links the asymptotic behaviour of this norm with that of an infinite product of matrices indexed on the binary expansion of $s$.

In section \ref{sec:rational}, we answer the second part of the question: we present an algorithm that allows to calculate the derivative of $u$ (hence also that of $\trifun{u}{a}{c}{b}$) in every rational point.

In section \ref{sec:derivative}, we answer the question in the general case: we show that the derivative of $u$, when it exists, can only take the values $0$ or $\infty$, and that it is almost surely equal to $0$.

Finally, in section \ref{sec:numerical}, we give some numerical results, which help us make a few estimations. We establish a simple sufficient condition for the derivative to be equal to $0$, and we conjecture one for the derivative to be equal to $\infty$.

\subsection{Acknowledgements}
\label{sec:thanks}
I would like to thank Mr. Yves Benoist, my master's degree advisor, who helped me a lot with this work.

\section{Preliminaries : harmonic function on the Sierpinski triangle}
\label{sec:prelim}

\begin{defn}
Let us first define successive approximations of the Sierpinski triangle: we start with
\[\mathcal{S}_0 := \{0, 1, \omega\},\]
and we define recursively, for all $n \in \mathbb{N}$,
\[\mathcal{S}_{n+1} := h_0(\mathcal{S}_n) \cup h_1(\mathcal{S}_n) \cup h_\omega(\mathcal{S}_n).\]
Let us give a name to the union of all these approximations:
\[\mathcal{S}_\infty := \bigcup_{n=0}^{+\infty}\mathcal{S}_n,\]
and the \emph{Sierpinski triangle} is then defined as the closure of this union:
\[\mathcal{S} := \overline{\mathcal{S}_\infty}.\]
\end{defn}

\begin{figure}
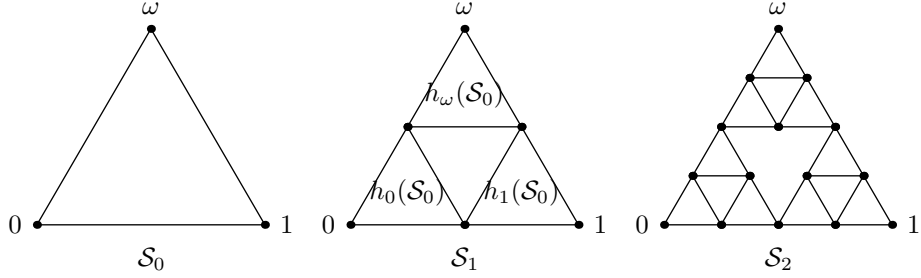

  \captionsetup[subfigure]{font=normalsize, labelformat=empty}
  \centering
  \subfloat[$\mathcal{S}_0$]{\includegraphics{fig1-0.1}} \quad
  \subfloat[$\mathcal{S}_1$]{\includegraphics{fig1-1.1}} \quad
  \subfloat[$\mathcal{S}_2$]{\includegraphics{fig1-2.1}}
  \caption{Discrete approximations of Sierpinski triangle.}
  \label{fig:discappr}
\end{figure}

Note than we can naturally see $\mathcal{S}_n$ as a graph: $\mathcal{S}_0$ is a complete graph on three vertices, and as edges of $\mathcal{S}_{n+1}$, we take the images of the edges of $\mathcal{S}_n$ by the three homotheties. Then it is easy to see that every vertex except $0$, $1$ and $\omega$ has exactly four neighbours. (See Figure \ref{fig:discappr}.) In the following, $E$ is a real vector space, and we will write ``$s \sim t$'' for ``$s$ is a neighbour of $t$''.
\begin{defn}
Let $n \in \mathbb{N}$. A \emph{harmonic function} on $\mathcal{S}_n$ is a function $f: \mathcal{S}_n \to E$ such that:
\[\forall s \in \mathcal{S}_n \setminus \{0, 1, \omega\},\quad   f(s) = \frac{1}{4}\sum_{t \sim s}f(t).\]
\end{defn}
We impose no condition on the values $f(0)$, $f(1)$ and $f(\omega)$: we consider them to be boundary conditions.

% One might ask oneself why we consider the ``boundary'' of $\mathcal{S}_n$ to be only these three points. Here is a heuristic explanation: when you split, for example, a square in four smaller squares, the smaller squares join together along their sides, so the boundary of a square is formed by its sides. However, when you split a Sierpinski triangle in three smaller copies of itself, those copies only join together along their corners, and never along their sides, so it is natural to say that the ``boundary'' of a Sierpinski triangle is formed by its three corners.

\begin{defn}
A \emph{harmonic function} on $\mathcal{S}$ is a continuous function $f: \mathcal{S} \to E$ whose restriction to every $\mathcal{S}_n$ is harmonic.
\end{defn}
\begin{propdefn}
\label{Definition of fabc}
Let $a, b, c \in E$. Then there exists a unique function $f$ that is harmonic on $\mathcal{S}$ and that satisfies the boundary conditions:
\[f(0) = a,\quad   f(1) = b,\quad   f(\omega) = c.\]
(This means in particular that the set of all real-valued harmonic functions on $\mathcal{S}$ is a three-dimensionnal vector space.) We will call this function $\trifun{f}{a}{c}{b}$.
\end{propdefn}
\begin{proof}
To check this, we will explicitly construct, by induction, the family of functions that verify these conditions on successive approximations of $\mathcal{S}$.

First we define the (obviously unique) harmonic function on $\mathcal{S}_0 = \{0, 1, \omega\}$ that satisfies the boundary conditions:
\[f_0:\quad 0 \mapsto a,\;  1 \mapsto b,\;  \omega \mapsto c.\]

\begin{figure}
\centering
\includegraphics{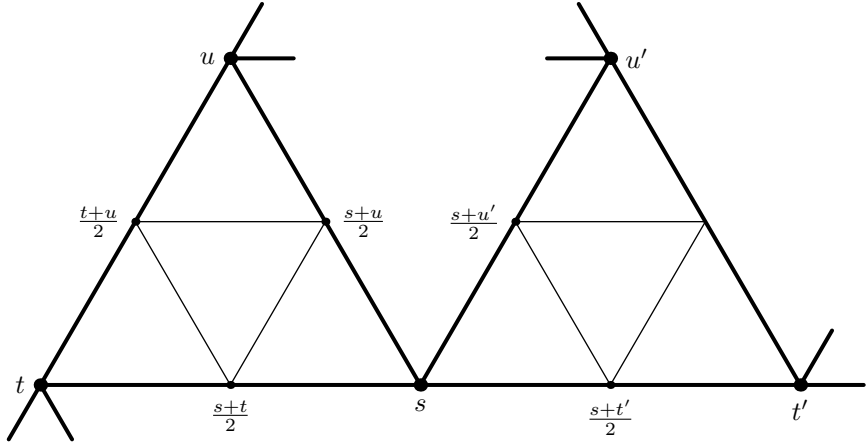}
\caption{Subdivision of $\mathcal{S}_n$.}
\label{fig:subdiv}
\end{figure}

Now suppose we have a harmonic function $f_n: \mathcal{S}_n \to E$ satisfying the boundary conditions. We will now define a function $f_{n+1}$ that extends $f_n$ to $\mathcal{S}_{n+1}$. Let $s, t, u$ be any three adjacent vertices of $\mathcal{S}_n$ that form a triangle having the same orientation as $\mathcal{S}_0$. Then $\mathcal{S}_{n+1}$ still contains these vertices, but also the vertices $\frac{s+t}{2}, \frac{t+u}{2}, \frac{s+u}{2}$ (see Figure \ref{fig:subdiv}). Conversely, it is easy to see that all vertices of $\mathcal{S}_{n+1}$ can be accounted for in this way. We then set
\begin{align*}
f_{n+1}\left(\tfrac{s+t}{2}\right) &= \frac{2f_n(s) + 2f_n(t) + f_n(u)}{5} \\
f_{n+1}\left(\tfrac{s+u}{2}\right) &= \frac{2f_n(s) + f_n(t) + 2f_n(u)}{5} \\
f_{n+1}\left(\tfrac{t+u}{2}\right) &= \frac{f_n(s) + 2f_n(t) + 2f_n(u)}{5} \\
\end{align*}
Being an extension of $f_n$, this function still satisfies the boundary conditions. Checking that it is harmonic on $\mathcal{S}_{n+1}$ is a straightforward calculation; it is enough to do it in $s$ and in $\frac{s+t}{2}$, since any vertex is similar to one of those two.
\begin{itemize}
\item The neighbours of $s$ in $\mathcal{S}_{n+1}$ are $\frac{s+t}{2}$, $\frac{s+u}{2}$, $\frac{s+t'}{2}$ and $\frac{s+u'}{2}$, where $t'$ and $u'$ are the two other neighbours of $s$ in $\mathcal{S}_n$ (see Figure \ref{fig:subdiv}). We have:
\begin{align*}
&f_{n+1}(\tfrac{s+t}{2}) + f_{n+1}(\tfrac{s+u}{2}) + f_{n+1}(\tfrac{s+t'}{2}) + f_{n+1}(\tfrac{s+u'}{2}) = \\
\\
&= \frac{2f_n(s) + 2f_n(t) + f_n(u)}{5}\; +\; \frac{2f_n(s) + f_n(t) + 2f_n(u)}{5}\; +\\
&\qquad +\; \frac{2f_n(s) + 2f_n(t') + f_n(u')}{5}\; +\; \frac{2f_n(s) + f_n(t') + 2f_n(u')}{5} \\
&= \frac{8f_n(s)\; +\; 3\big(f_n(t) + f_n(u) + f_n(t') + f_n(u')\big)}{5} \\
&= \frac{8+12}{5}f_n(s) \\
&= 4f_{n+1}(s) \\
\end{align*}
\item The neighbours of $\frac{s+t}{2}$ are $s$, $\frac{s+u}{2}$, $\frac{t+u}{2}$ and $t$. We have:
\begin{align*}
&f_{n+1}(s) + f_{n+1}(\tfrac{s+u}{2}) + f_{n+1}(\tfrac{t+u}{2}) + f_{n+1}(t) = \\
\\
&= f_n(s) + \frac{2f_n(s) + f_n(t) + 2f_n(u)}{5} + \frac{f_n(s) + 2f_n(t) + 2f_n(u)}{5} + f_n(t) \\
&= \frac{8f_n(s) + 8f_n(t) + 4f_n(u)}{5} \\
&= 4f_{n+1}(\tfrac{s+t}{2}) \\
\end{align*}
\end{itemize}
Now let us show that these functions are unique. Indeed, consider two harmonic functions on $\mathcal{S}_n$ that satisfy the same boundary conditions. This means that their difference vanishes at $0$, $1$ and $\omega$, and is harmonic as well. Then the maximum principle guarantees that it is identically equal to zero.

This allows us to define $\trifun{f}{a}{c}{b}$ on $\mathcal{S}_\infty$, and the above calculation gives us an explicit recurrence relation to compute it in any point. Let us rewrite this relation in a more convenient way:
\begin{equation}\label{eq:deffacb}
\begin{cases}
\trifun{f}{a}{c}{b} \circ h_0 =      \trifun{f}{a                }{\frac{2a+b+2c}{5}}{\frac{2a+2b+c}{5}} \\
\\
\trifun{f}{a}{c}{b} \circ h_1 =      \trifun{f}{\frac{2a+2b+c}{5}}{\frac{a+2b+2c}{5}}{b                } \\
\\
\trifun{f}{a}{c}{b} \circ h_\omega = \trifun{f}{\frac{2a+b+2c}{5}}{c                }{\frac{a+2b+2c}{5}}. \\
\end{cases}
\end{equation}
To extend this function continuously to the whole Sierpinski triangle, we need to check that it is uniformly continuous. Let
\[D(a, b, c)\; :=\; \maxof(\|a-b\|,\; \|b-c\|,\; \|a-c\|).\]
By the maximum principle, we have, for all values of $a, b, c$ and for all $s, t \in \mathcal{S}_\infty$,
\[\|\trifun{f}{a}{c}{b}(s) - \trifun{f}{a}{c}{b}(t)\|\; \leq\; D(a, b, c).\]
On the other hand, it is easy to check that we have
\[D(a, \tfrac{2a+b+2c}{5}, \tfrac{2a+2b+c}{5})\; \leq\;
\tfrac{3}{5}D(a, b, c).\]
Using \eqref{eq:deffacb}, by induction, il follows that if $T$ is an ``elementary triangle'' of level $n$, we have
\[\forall s, t \in T,\quad   \|\trifun{f}{a}{c}{b}(s) - \trifun{f}{a}{c}{b}(t)\|\; \leq\; {(\tfrac{3}{5})}^n C,\]
where $C$ depends only on $a$, $b$ and $c$. Now consider two points of $\mathcal{S}_\infty$ whose distance is at most ${(\tfrac{1}{2})}^n$; then it is easy to see that the respective elementary triangles of level $n-1$ that contain them are either coincident or adjacent. It follows that
\[\forall s, t \in \mathcal{S}_\infty,\quad   |s - t|\; \leq\; {(\tfrac{1}{2})}^n \implies \|\trifun{f}{a}{c}{b}(s) - \trifun{f}{a}{c}{b}(t)\|\; \leq\; {(\tfrac{3}{5})}^n C',\]
where $C'$ is still some positive real constant that depends only on $a$, $b$ and $c$ --- which leads to the conclusion. This shows that $\trifun{f}{a}{c}{b}$ is indeed well-defined on $\mathcal{S}$.
\end{proof}
Very informally, $\trifun{f}{a}{c}{b}$ describes the shape that a Sierpinski triangle made of rubber would assume if it were stretched between three nails fixed in $a$, $b$ and $c$ (see Figure \ref{fig:harmsierp}).

\begin{figure}
\noindent\makebox[\textwidth]{\includegraphics{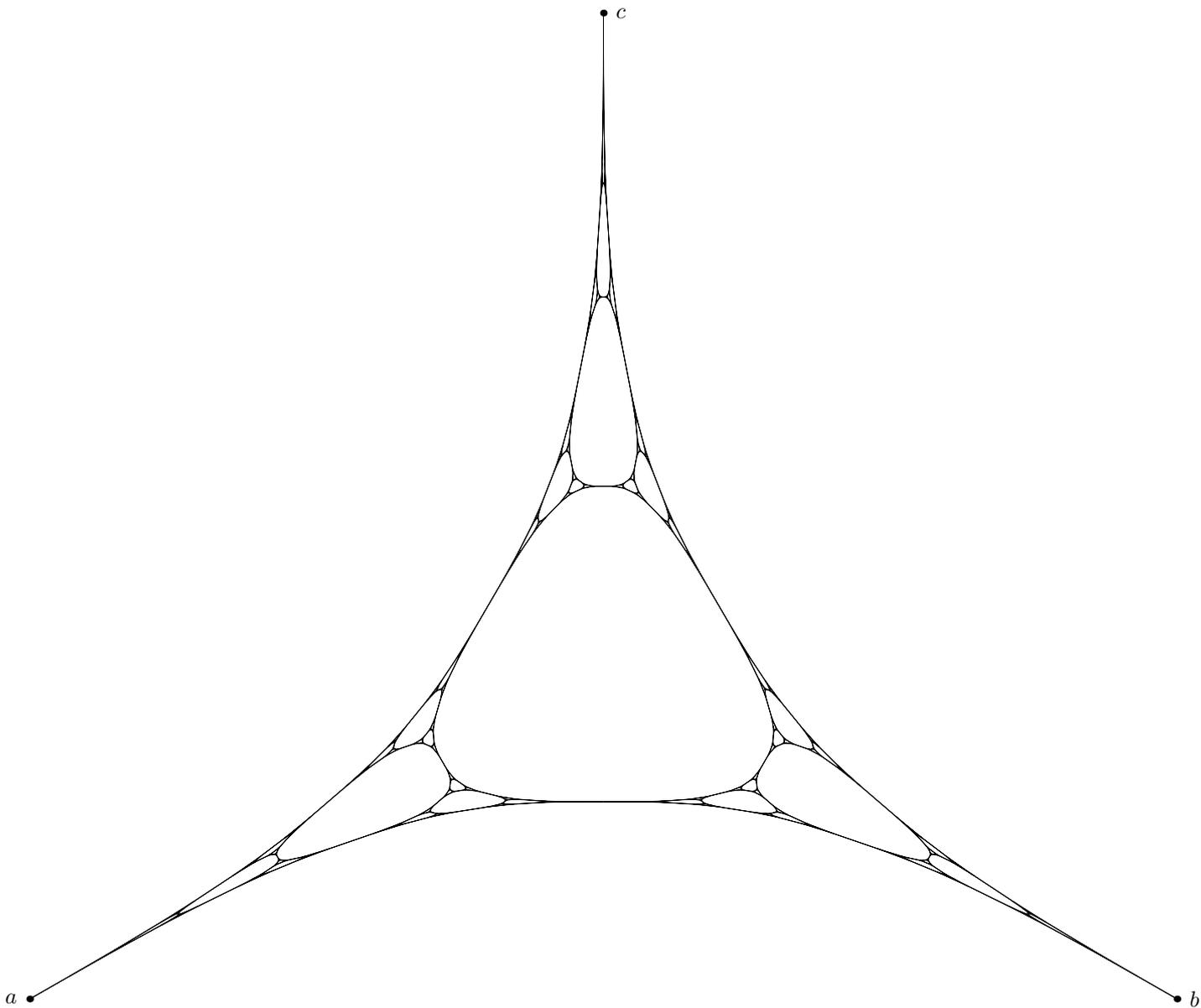}}
\caption{Harmonic image of the Sierpinski triangle, $\trifun{f}{a}{c}{b}(\mathcal{S})$.}
\label{fig:harmsierp}
\end{figure}

Note that by continuity, the relations \eqref{eq:deffacb} are still valid in $\mathcal{S}$. They are very important: they supply an intrinsic definition of the functions $\trifun{f}{a}{c}{b}$.

From now on, following Kirillov, we restrict ourselves to a side of the Sierpinski triangle: we let $\trifun{u}{a}{c}{b} = {\left. \trifun{f}{a}{c}{b} \right|}_{[0,1]}$. We may then drop the third line from \eqref{eq:deffacb}. Moreover, Kirillov is mainly interested in real-valued functions, in particular four of them: $\phi = \trifun{u}{0}{0}{1}$, $\psi = \trifun{u}{0}{1}{1}$, $\chi = \trifun{u}{0}{-1}{1}$ and $\xi = \trifun{u}{0}{2}{1}$ (in his notations, $u$ actually stands for some one function among these four). However, they have a drawback: it is difficult to study one of this functions independently from others, because \eqref{eq:deffacb} forces us to change the values of $a$, $b$ and $c$.

This leads us to introduce the function $\trifun{u}{e_0}{e_\omega}{e_1}$, where $(e_0,e_1,e_\omega)$ is the canonical basis of $\mathbb{R}^3$. Its main interest is that it contains all the information about all the functions $\trifun{u}{a}{c}{b}$. Indeed, for all $a$, $b$ and $c$, we have $\trifun{u}{a}{c}{b} = \Phi \circ \trifun{u}{e_0}{e_\omega}{e_1}$, where $\Phi:\;\mathbb{R}^3 \to E$ is the linear map that send the canonical basis to $(a,b,c)$. (If $a$, $b$ and $c$ are real, it is simply the linear form with matrix $\begin{pmatrix}a & b & c\end{pmatrix}$.) The relations \eqref{eq:deffacb} then give us a functional equation on $\trifun{u}{e_0}{e_\omega}{e_1}$ --- which is none other than \eqref{eq:recrelu}. Thus we see that the function $u$ that we have defined in the beginning is just a shorter notation (that we shall adopt from now on) for $\trifun{u}{e_0}{e_\omega}{e_1}$. Its existence and unicity follow from the proof of Proposition-Definition \ref{Definition of fabc}.

\section{General properties of $u$}
\label{sec:general}

Note that all that we shall say in this section can very easily be generalized to all of $\mathcal{S}$, that is to $\trifun{f}{e_0}{e_\omega}{e_1}$. To do this, it is enough to introduce the matrix $M_\omega$, analogous to $M_0$ and $M_1$, and to see that a generic point of $\mathcal{S}$ can be described by an infinite sequence of symbols $0$, $1$ or $\omega$, by analogy with the binary expansion.

We will need to diagonalise the $M_i$. Any of these matrices has eigenvalues $1$, $\frac{3}{5}$ and $\frac{1}{5}$, with respective eigenvectors $e_i$, $\vec{v}_i$ and $\vec{w}_i$, where:
\begin{align*}
e_0 = \begin{pmatrix}1\\0\\0\end{pmatrix},\quad
\vec{v}_0 = \begin{pmatrix}-1\\\frac{1}{2}\\\frac{1}{2}\end{pmatrix},\quad
\vec{w}_0 = \begin{pmatrix}0\\\frac{1}{2}\\-\frac{1}{2}\end{pmatrix}, \\
\\
e_1 = \begin{pmatrix}0\\1\\0\end{pmatrix},\quad
\vec{v}_1 = \begin{pmatrix}-\frac{1}{2}\\1\\-\frac{1}{2}\end{pmatrix},\quad
\vec{w}_1 = \begin{pmatrix}-\frac{1}{2}\\0\\\frac{1}{2}\end{pmatrix}.
\end{align*}

Note that
\[\trifun{f}{1}{0}{0} + \trifun{f}{0}{0}{1} + \trifun{f}{0}{1}{0} = \trifun{f}{1}{1}{1} \equiv 1,\]
which shows that our picture is in fact only two-dimensionnal: all the values of $\trifun{f}{e_0}{e_\omega}{e_1}$, hence also all the values of $u$, actually lie in the affine plane
\[\mathcal{H} := \setsuch{(x, y, z) \in \mathbb{R}^3}{x+y+z = 1}.\]
Of course, it is stable by the maps $M_i$, and the latter induce some affine maps on the former (see Figure \ref{fig:mi}). In fact, this is why we have chosen to write $e_i$ instead of $\vec{e}_i$: the $e_i$ are elements of the affine plane $\mathcal{H}$, namely --- under this interpretation --- fixed points of the affine maps $M_i$.

\begin{figure}[!h]
\centering
\includegraphics{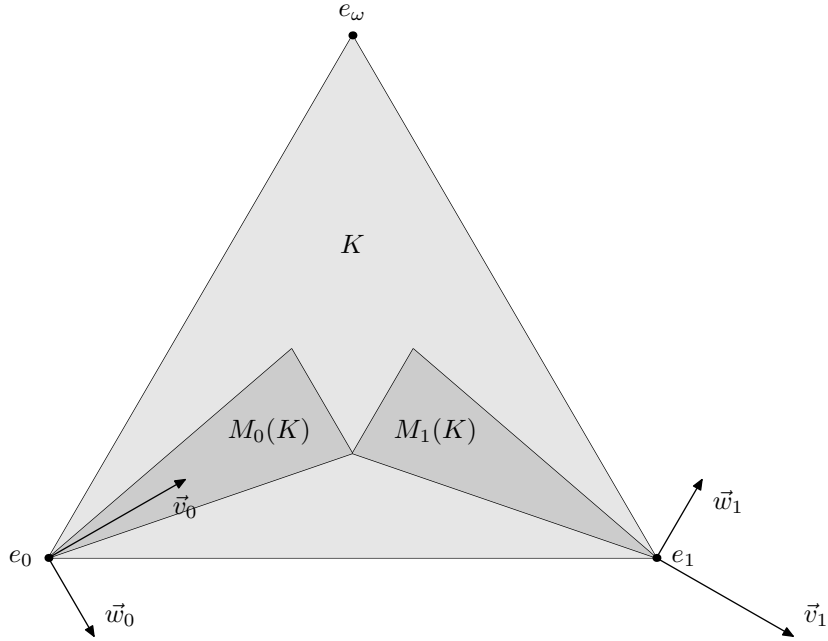}
\caption{Eigenvectors of the operators $M_i$; the plane of the sheet of paper corresponds to the affine plane $\mathcal{H}$. Thus the origin of $\mathbb{R}^3$ lies beneath the centre of the figure. The triangle $K$ is defined a little bit further; it is added here to better illustrate the action of the $M_i$ on $\mathcal{H}$.}
\label{fig:mi}
\end{figure}

Let us now try to write down an explicit formula for $u(s)$. We first need to express $s$ in terms of $0$, of $1$ and of the $h_i$, which may be done by writing:
\[s = \lim_{n \to \infty}h_{a_1 \ldots a_n}(0) = \lim_{n \to \infty}h_{a_1 \ldots a_n}(1),\]
where $s = 0,a_1a_2\ldots$ is a (lower or upper) binary expansion of $s$. Using the above relations, we get the following:
\begin{prop}
\label{u general formula}
Let $K := \mathcal{H} \cap \mathbb{R}_{\geq 0}^3$ be the (full) triangle with the canonical coordinate vectors as vertices. Then for all $s \in [0, 1]$ and for any initial value $u_0 \in K$, we have
\begin{equation}\label{eq:formulau}
u(s) = \lim_{n \to \infty}M_{a_1 \ldots a_n}(u_0),
\end{equation}
where $s = 0.a_1a_2 \ldots$ is its (lower or upper) binary expansion.
\end{prop}
(See Figure \ref{fig:ut}.)

\begin{figure}
\centering
\includegraphics{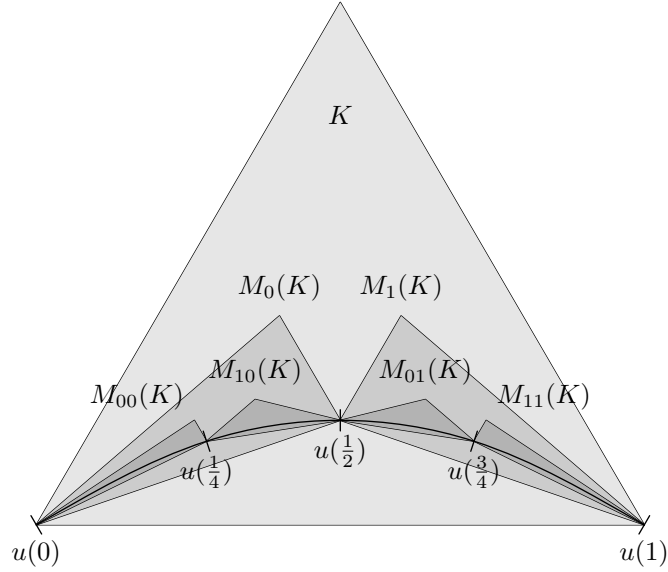}
\caption{Illustration of Proposition \ref{u general formula}. The curve that runs through the picture is the parametric curve $u(t)$.}
\label{fig:ut}
\end{figure}

\begin{proof}
Let us first rewrite this formula in such a way as to make it clear that the result does not depend on the initial value. It is easy to check that $K$ is stable by both $M_0$ and $M_1$: thus the sequence $M_{a_1 \ldots a_n}(K)$ is decreasing. In the proof of Proposition-Definition \ref{Definition of fabc}, we have already seen that the diameters of these sets tend to 0. It is clear that they are closed. Thus their intersection contains a single point:
\[\forall u_0 \in K,\quad  
\left\{ \lim_{n \to \infty}M_{a_1 \ldots a_n}(u_0) \right\} = \bigcap_{n=0}^{\infty}M_{a_1 \ldots a_n}(K).\]
On the other hand, a trivial induction shows that for any finite word $w$ on the alphabet $\{0, 1\}$, $h_{w}(0) = 0.w$, where $0.w$ is seen as a binary fraction, and that $u \circ h_{w} = M_{w} \circ u$. Hence, by the relations \eqref{eq:recrelu} and by continuity of $u$,
\[u\left(\sum_{i=1}^{\infty}a_i2^{-i}\right)
= \lim_{n \to \infty}M_{a_1 \ldots a_n} \begin{pmatrix}1\\0\\0\end{pmatrix},\]
and hence the formula.\end{proof}

\begin{cor}
\label{u injective}
$u$ is injective.
\end{cor}

\begin{proof}Let $(a_i)_{i \geq 1}$ and $(a'_i)_{i \geq 1}$ be two sequences with values in $\{0, 1\}$. Let $k$ be the first index where these sequences diverge, i. e. suppose $a_1 = a'_1,\; \ldots,\; a_{k-1} = a'_{k-1},\; a_k \neq a'_k$. Without loss of generality, we may actually suppose that $a_k = 0$ and $a'_k = 1$. Choose some $u_0 \in K$. Then the following statements are equivalent (remember that the $M_i$ are bijective):
\begin{align*}
\lim_{n \to \infty}M_{a_1 \ldots a_n}(u_0) &= \lim_{n \to \infty}M_{a'_1 \ldots a'_n}(u_0)\\
\bigcap_{n=0}^{\infty}M_{a_1 \ldots a_n}(K) &= \bigcap_{n=0}^{\infty}M_{a'_1 \ldots a'_n}(K)\\
\bigcap_{n=k}^{\infty}M_{0 a_{k+1} \ldots a_n}(K) &= \bigcap_{n=k}^{\infty}M_{0 a'_{k+1} \ldots a'_n}(K)\\
\end{align*}
On the last line, each side is a subset of the respective $M_i(K)$. But a quick calculation (or a quick look at Figure \ref{fig:ut}) shows that the two images of the triangle intersect at a single point:
\[M_0(K) \cap M_1(K) = \left\{\frac{1}{5}\begin{pmatrix}2\\2\\1\end{pmatrix}\right\}
= \{M_0(e_1)\} = \{M_1(e_0)\}.\]
Thus the previous equality holds if and only if
\[
\begin{cases}
\displaystyle \bigcap_{n=k}^{\infty}M_{a_{k+1} \ldots a_n}(K) = \{e_1\} \\
\displaystyle \bigcap_{n=k}^{\infty}M_{a'_{k+1} \ldots a'_n}(K) = \{e_0\}, \\
\end{cases}
\]
which is true if and only if for all $i > k$, $a_i = 1$ and $a'_i = 0$. In other words, both sequences yield the same value if and only if we have $\sum_{i=1}^{\infty}a_i2^{-i} = \sum_{i=1}^{\infty}a'_i2^{-i}$. The ``only if'' part proves our claim, and the ``if'' part offers additional support for the equivalence of different binary expansions.
\end{proof}

\section{Local behavior of $u$: existence and direction of tangent}
\label{sec:direction}

We will now study the behaviour of the vector $u(t) - u(s)$ ad $t \to s$. In this section, we will thoroughly describe the behaviour of the norm of this vector: quantitatively through the formula \eqref{eq:directionformula}, and qualitatively in Corollary \ref{properties of Du}.

The difference $u(t) - u(s)$ lives in $\vec{\mathcal{H}} := \setsuch{(x, y, z) \in \mathbb{R}^3}{x+y+z = 0}$, the vector plane parallel to the affine plane $\mathcal{H}$. This vector plane is of course stable by $M_0$ and $M_1$; it is thus natural to introduce the restrictions of these operators to $\vec{\mathcal{H}}$, that we shall call $\vec{M}_0$ and $\vec{M}_1$. Each $\vec{M}_i$ has eigenvalues $\frac{3}{5}$ and $\frac{1}{5}$, with respective eigenvectors $\vec{v}_i$ and $\vec{w}_i$. Note that in both cases, $\vec{v}_i$ and $\vec{w}_i$ actually form an orthogonal basis of $\mathcal{H}$.

\begin{figure}
\centering
\includegraphics{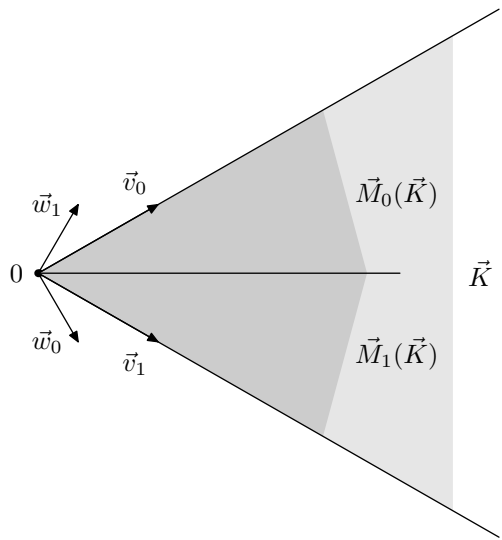}
\caption{Main objects of the vector plane $\vec{\mathcal{H}}$.}
\label{fig:vech}
\end{figure}

Let us now introduce the set
\[\vec{K} := \setsuch{a\vec{v}_0 + b\vec{v}_1}{a, b \geq 0 \text{ and } (a, b) \neq (0, 0)}\]
(see Figure \ref{fig:vech}.) This notation should not lead the reader to think that this set is somehow derived from $K$: it is not, at least not directly. But it does play an analogous role, and the following lemma will show precisely in what sense:

\begin{lem}
\label{K stable by m}
The set $\vec{K}$ is stable by the $\vec{M}_i$. For all $s, t \in [0, 1]$ such that $s < t$, $u(t) - u(s) \in \vec{K}$.
\end{lem}
\begin{proof}
The first part can be checked by a straightforward calculation (see also Figure \ref{fig:vech}).

As for the second part, let us first prove the result for dyadic rational values of $s$ and $t$ (such that $0 \leq s < t \leq 1$), by induction on their dyadic valuations. Let us write $I_n := \mathcal{S}_n \cap [0, 1] = \setsuch{\frac{k}{2^n}}{0 \leq k \leq 2^n}$, so that we have $I_0 = \{0, 1\}$ and $I_{n+1} = h_0(I_n) \cup h_1(I_n)$. If $s, t \in I_0$, then $s = 0$, $t = 1$ and the conclusion can be checked by a straightforward calculation. Now suppose that the conclusion is true for all $s, t \in I_n$; let $s < t$ be two points in $I_{n+1}$. We need to distinguish three cases: $s < t \leq \frac{1}{2}$, $s < \frac{1}{2} < t$ and $\frac{1}{2} \leq s < t$. We shall only treat the middle case, which is the hardest one; the other two are analogous. We may write $s = h_0(S)$ and $t = h_1(T)$, so that we have:
\begin{align*}
u(t) - u(s)\; &=\; u(t) - u(\tfrac{1}{2})\; +\; u(\tfrac{1}{2}) - u(s) \\
              &=\; u\big(h_1(T)\big) - u\big(h_1(0)\big)\; +\; u\big(h_0(1)\big) - u\big(h_0(S)\big) \\
              &=\; \vec{M}_1\big(u(T) - u(0)\big)\; +\; \vec{M}_0\big(u(1) - u(S)\big). \\
\end{align*}
Applying, in order, the induction hypothesis, the fact that $\vec{K}$ is stable by the $\vec{M}_i$, and the fact that $\vec{K}$ is stable under linear combination with positive coefficients (which follows immediately from the formula that defines it), we get the result.

Now let $0 \leq s < t \leq 1$ be general values. Since $u$ is continuous and since dyadic rationals are dense, $u(t) - u(s)$ lies in the closure of $\vec{K}$, which is equal to $\vec{K} \cup \{0\}$. But since, by Corollary \ref{u injective}, $u$ is injective, the difference cannot be equal to zero.
\end{proof}

We now need to introduce the projective equivalents of the relevant objects. Let $\pi: \vec{\mathcal{H}} \setminus \{0\} \to \mathbb{P}^1(\mathbb{R})$ be the canonical projection; we let $\tilde{K} := \pi(\vec{K})$ and
\begin{align*}
\tilde{M}_i:\quad \mathbb{P}^1(\mathbb{R}) &\to     \mathbb{P}^1(\mathbb{R}) \\
                  \pi(\vec{x})             &\mapsto \pi(M_i(\vec{x}))
\end{align*}
(see Figure \ref{fig:tilde}). Clearly, $\tilde{K}$ is stable by the $\tilde{M}_i$.

\begin{figure}
\centering
\includegraphics{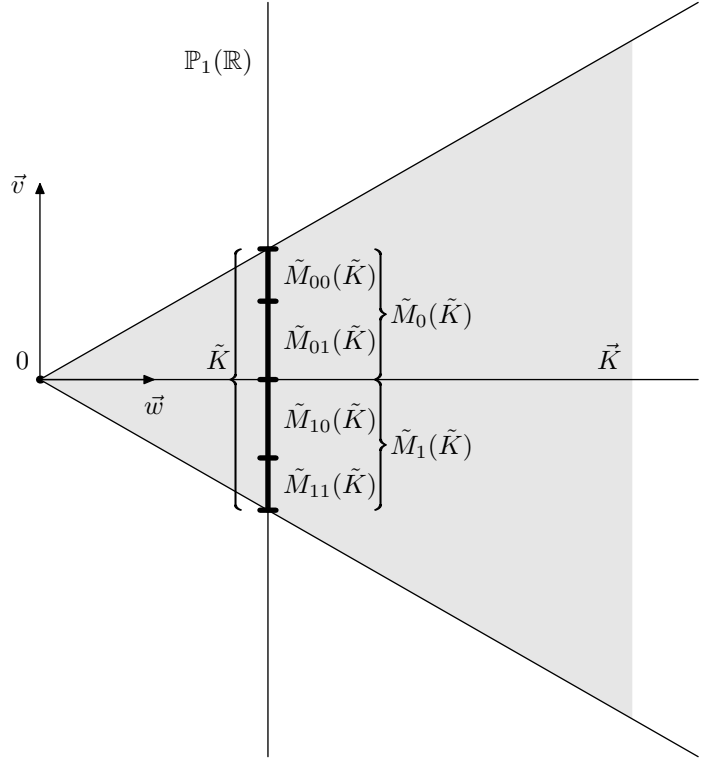}
\caption{Illustration of Proposition \ref{tangentformula}}
\label{fig:tilde}
\end{figure}

\begin{prop}
\label{tangentformula}
Let $s \in [0, 1]$; let $s = 0,a_1a_2 \ldots$ be its upper binary expansion. Then for all $\tilde{u}_0 \in \tilde{K}$, we have
\begin{equation}\label{eq:directionformula}
\lim_{\substack{t \to s \\ t > s}}\; \pi\big(u(t) - u(s)\big)
= \lim_{n \to \infty} \tilde{M}_{a_1 \ldots a_n} (\tilde{u}_0),
\end{equation}
that is, both limits always exist and are always equal.

As for lower expansions, the formula becomes:
\[\lim_{\substack{t \to s \\ t < s}}\; \pi\big(u(s) - u(t)\big)
= \lim_{n \to \infty} \tilde{M}_{a_1 \ldots a_n} (\tilde{u}_0).\]
\end{prop}
\begin{proof}
The proof is quite similar to the proof of Proposition \ref{u general formula}. Just as before, the sets $\tilde{M}_{a_1 \ldots a_n}(\tilde{K})$ are all closed and form a decreasing sequence. To ensure that their intersection contains a single point, it is enough to check that their diameters tend to 0; to check this, it is enough to show that the $\tilde{M}_i$ are contractive. In the usual metric (where the distance between two points is the angle between the corresponding lines), this is not the case: the Lipschitz constant is equal to 1. But we shall use another metric --- the one defined by $d\big(\pi(\vec{u}_1), \pi(\vec{u}_2)\big) = |\frac{a_1}{b_1} - \frac{a_2}{b_2}|$, where we let $\vec{u}_i = a_i\vec{v} + b_i\vec{w}$ with $\vec{v} = \begin{pmatrix}-\frac{1}{2}\\-\frac{1}{2}\\1\end{pmatrix}$ and $\vec{w} = \begin{pmatrix}-\frac{1}{2}\\\frac{1}{2}\\0\end{pmatrix}$. Informally, what we do is to project on a line rather than on a circle: see Figure \ref{fig:tilde}. Since there exists a neighbourhood of the point at infinity which does not intersect $\tilde{K}$, this metric is indeed defined on $\tilde{K}$ and induces the right topology on it. Then a few lines of calculation show that in this metric, both $\tilde{M}_0$ and $\tilde{M}_1$ have Lipschitz constant $\frac{3}{4}$. Hence we can write
\[\forall \tilde{u}_0 \in \tilde{K},\quad  
\left\{ \lim_{n \to \infty}\tilde{M}_{a_1 \ldots a_n}(\tilde{u}_0) \right\} = \bigcap_{n=0}^{\infty}\tilde{M}_{a_1 \ldots a_n}(\tilde{K}).\]
Now let $n \in \mathbb{N}$. Then for any $t$ sufficiently close to, but greater than, $s$, the first $n$ bits of the binary expansions of $t$ and $s$ coincide (this is true because we use the upper binary expansion of $s$), so that we have $s = h_{a_1 \ldots a_n}(S)$ and $t = h_{a_1 \ldots a_n}(T)$ for some $S, T \in [0, 1]$. By Lemma \ref{K stable by m}, we have
\[\pi\big(u(t) - u(s)\big) = \tilde{M}_{a_1 \ldots a_n}\big(\pi \big(u(T) - u(S)\big)\big) \in \tilde{M}_{a_1 \ldots a_n}(\tilde{K}),\]
hence the result.

As for the symmetric case, the proof is completely analogous.
\end{proof}

\begin{cor}
\label{properties of Du}
The quantity
\[\vec{D}u(s) := \lim_{\substack{t \to s \\ t \in [0, 1]}}\pi\big(u(t) - u(s)\big)\sgn(t - s)\]
is a well-defined, injective, continuous function of $s$.
\end{cor}
In geometric terms, $\vec{D}u(s)$ is the direction of the tangent to the curve which is the \emph{image} of $u$. This result may be interpreted as saying that this curve is $C^1$ and convex --- which is after all not so surprising if we think of it as the shape of the side of an elastic Sierpinski triangle stretched on three nails (see Figure \ref{fig:harmsierp}). Note that the \emph{parametrization} of the curve is much less regular than that --- as we shall see, it is not even always differentiable.
\begin{proof}
Being given two sequences $(a_i)_{i \geq 1}$ and $(a'_i)_{i \geq 1}$ with values in $\{0, 1\}$, we claim that
\[\bigcap_{n=0}^{\infty}\tilde{M}_{a_1 \ldots a_n}(\tilde{K}) = \bigcap_{n=0}^{\infty}\tilde{M}_{a'_1 \ldots a'_n}(\tilde{K})\]
if and only if both sequences are binary expansions of the same number. The proof simply follows, \emph{mutatis mutandis}, the proof of Corollary \ref{u injective}; we need the fact that the $\tilde{M}_i$ are injective and that $\tilde{M}_0(\tilde{K})$ and $\tilde{M}_1(\tilde{K})$ intersect at a single point, which are easily seen to be true. Combined with the previous Proposition, the ``if'' part shows that the function is well-defined and the ``only if'' part shows that it is injective.

As for the continuity, let $\eps > 0$. Since $\tilde{M}_0$ and $\tilde{M}_1$ are contractive (in a suitable metric), we can find a number $n$ such that the diameter of every possible $\tilde{M}_{a_1 \ldots a_n}(\tilde{K})$ is less than $\frac{\eps}{2}$. Now let $s, s' \in [0, 1]$ such that $|s - s'| \leq 2^{-n}$. Let $s = 0.a_1 \ldots a_n \ldots$ and $s' = 0.a'_1 \ldots a'_n \ldots$ be their binary expansions; let $w = a_1 \ldots a_n$, so that $\vec{D}u(s) \in \tilde{M}_w(\tilde{K})$, and similarly define $w'$. Then two cases are possible: either $w = w'$, or $w$ and $w'$ are two consecutive words. However, knowing that $\tilde{M}_0\big(\pi(\vec{w}_0)\big) = \pi(\vec{w}_0)$, $\tilde{M}_1\big(\pi(\vec{w}_1)\big) = \pi(\vec{w}_1)$ and $\tilde{M}_1\big(\pi(\vec{w}_0)\big) = \tilde{M}_0\big(\pi(\vec{w}_1)\big)$, a trivial induction shows that in the latter case, $\tilde{M}_w(\tilde{K})$ and $\tilde{M}_w(\tilde{K})$ have a common point. Hence the distance between $\vec{D}u(s)$ and $\vec{D}u(s')$ is less than $\eps$, which finishes the proof.
\end{proof}

\section{Local behavior of $u$ and $\trifun{u}{a}{c}{b}$: rate of change}
\label{sec:norm}
\subsection{Hölder exponent formula}
\label{sec:hölderexpformula}

In the previous section, we have seen that as $t$ tends to $s$, the direction of $u(t) - u(s)$ behaves rather well. We shall now see that its norm behaves much more erratically. A naive approach would be to simply try to calculate the derivative $u'(s)$; however, we shall see in Proposition \ref{no finite derivative} that whenever this derivative exists, it is equal to either $0$ or $\infty$, so it is not the most relevant parameter here.

We introduce the following parameters, that will provide a better description of the local behavior of $u$:
\begin{defn}
Let $s \in [0, 1]$.
\begin{itemize}
\item The \emph{upper Hölder exponent} of $u$ in $s$ is
\[
\alpha_{\sup}(u, s) := \limsup_{\substack{t \to s \\ t \in [0, 1]}}\frac{\ln \|u(t) - u(s)\|}{\ln |t-s|}.
\]
\item The \emph{lower Hölder exponent} of $u$ in $s$ is
\[
\alpha_{\inf}(u, s) := \liminf_{\substack{t \to s \\ t \in [0, 1]}}\frac{\ln \|u(t) - u(s)\|}{\ln |t-s|}.
\]
\item When both are equal, we shall simply call them the \emph{Hölder exponent} of $u$ in $s$, and we shall write
\[
\alpha(u, s) := \lim_{\substack{t \to s \\ t \in [0, 1]}}\frac{\ln \|u(t) - u(s)\|}{\ln |t-s|}.
\]
\end{itemize}
\end{defn}
Once we know the Hölder exponent of $u$, we can usually calculate its derivative, and then the derivative of $\trifun{u}{a}{c}{b}$. Indeed:
\begin{lem}
\label{Hölder exponent and derivative}
Let $s \in [0, 1]$. Then we have :
\begin{itemize}
\item $\alpha_{\inf}(u, s) > 1 \implies u'(s) = 0      \implies \alpha_{\inf}(u, s) \geq 1;$
\item $\alpha_{\sup}(u, s) < 1 \implies u'(s) = \infty \implies \alpha_{\sup}(u, s) \leq 1,$
\end{itemize}
where the latter equality is shorthand for ``$\lim_{t \to s}\frac{\|u(t) - u(s)\|}{|t-s|} = \infty$''.
\end{lem}
\begin{proof}This is an easy exercise in real analysis.\end{proof}

\begin{lem}
\label{Link between u and uacb}
Let $\Phi = \begin{pmatrix}a & b & c\end{pmatrix}$ be a linear form, and let $s \in [0, 1]$ such that $\vec{D}u(s) \notin \ker \Phi$. Then we have:
\begin{itemize}
\item $(\trifun{u}{a}{c}{b})'(s) = 0$ iff $u'(s) = 0$;
\item $(\trifun{u}{a}{c}{b})'(s) = \infty$ iff $u'(s) = \infty$;
\item $(\trifun{u}{a}{c}{b})'(s)$ is undefined iff $u'(s)$ is undefined.
\end{itemize}
\end{lem}
\begin{proof}This is basically an application of the formula for the derivative of a composite function.\end{proof} 

Note that by Corollary \ref{properties of Du}, for a given value of $\Phi$, this lemma applies for all values of $s$ except at most one. In particular, it is easy to check that for the functions $\phi$, $\psi$, $\chi$ and $\xi$, introduced in \cite{kirillov}, this lemma holds with only two exceptions: $\chi'(0)$ and $\xi'(1)$ (and in \cite{kirillov}, it is shown that both are equal to $0$.)

Here is the formula that allows one to calculate the Hölder exponent:
\begin{prop}
\label{Hölder exponent general formula}
Let $s \in [0, 1]$, $s = 0.a_1a_2 \ldots$ its binary expansion. Then we have:
\begin{equation}\label{eq:Hölderexpformula}
% \alpha_{\inf,\, \sup}\left(u, s\right)
% = \operatorname*{lim\,inf,sup}_{n \to \infty}
%    \frac 
%    	{\ln \| \vec{M}_{a_1 \ldots a_n} \|}
%    	{n \ln \left(\frac{1}{2}\right)}.
\begin{gathered}
\alpha_{\inf}\left(u, s\right)
  = \liminf_{n \to \infty}
     \frac 
   	  {\ln \| \vec{M}_{a_1 \ldots a_n} \|}
   	  {n \ln \left(\frac{1}{2}\right)}; \\
\alpha_{\sup}\left(u, s\right)
  = \limsup_{n \to \infty}
     \frac 
   	  {\ln \| \vec{M}_{a_1 \ldots a_n} \|}
   	  {n \ln \left(\frac{1}{2}\right)}
\end{gathered}
\end{equation}
\end{prop}
(compare this with \eqref{eq:formulau} and \eqref{eq:directionformula}). Before proving it, we shall need the following lemma:
\begin{lem}
Let $f$ be a linear automorphism of $\mathbb{R}^2$, let $\mathcal{C}$ be a salient closed cone that is stable by $f$, and let $\mathcal{C'} \subset \interior \mathcal{C}$ such that $\mathcal{C'} \cup \{0\}$ is also a closed cone. Then for all $\vec{x} \in \mathcal{C'}$, the ratio
\[\frac{\|f(\vec{x})\|}{\|f\|\|\vec{x}\|}\]
is positively bounded by constants that depend only on $\mathcal{C}$ and $\mathcal{C'}$.
\end{lem}
\begin{proof}
Let us choose a basis in which $\mathcal{C}$ becomes the upper right quadrant. We may now write:
\begin{itemize}
\item $\mathcal{C} = \setsuch{(x, y) \in \mathbb{R}^2}{x \geq 0,\; y \geq 0}$;
\item $f = \begin{pmatrix}a & b\\c & d\end{pmatrix}$, where the coefficients $a$, $b$, $c$ and $d$ are all nonnegative;
\item $\mathcal{C'} = \setsuch{(x, y) \in \mathcal{C} \setminus \{0\}}{m \leq \frac{x}{y} \leq M}$, where $m$ and $M$ are some finite positive bounds.
\end{itemize}
Let $\vec{x} = (x, y) \in \mathcal{C'}$. Let us use the supremum norm: we have $\|\vec{x}\|_{\infty} = \maxof(x, y)$, $\|f\|_{\infty} = \maxof(a, b, c, d)$ and $\|f(\vec{x})\|_{\infty} = \maxof(ax+by,\; cx+dy)$. But we have, on the one hand
\[\maxof(ax+by,\; cx+dy)\; \leq\; 2\maxof(a, b, c, d)\maxof(x, y),\]
and on the other hand
\begin{align*}
\maxof(ax+by,\; cx+dy)\; &\geq\; \maxof(a, b, c, d)\minof(x, y)\\
                         &\geq\; \mu \maxof(a, b, c, d)\maxof(x, y),
\end{align*}
where $\mu = \minof(m, M^{-1})$ depends only on $\mathcal{C}$ and $\mathcal{C'}$ as announced.
\end{proof}
Let $\vec{K}' =
\setsuch{a\vec{w}_0 + b\vec{w}_1}{a, b \geq 0}$ (on Figure \ref{fig:vech}, it is the cone that lies between $\vec{w}_0$ and $\vec{w}_1$, not represented to avoid overloading the picture). Then it is easy to check that it is stable by $\vec{M}_0$ and $\vec{M}_1$, hence also by all of their products. By applying the lemma on $\mathcal{C} = \vec{K}'$ and $\mathcal{C'} = \vec{K}$ (watch out for the slight inconsistency in notations), we get the following estimation:
\begin{cor}
\label{Matrix-vector separation}
For any finite word $w$ on the alphabet $\{0, 1\}$ and for all $\vec{u}_0 \in \vec{K}$, the ratio
\[\frac{\|\vec{M}_{w}(\vec{u}_0)\|}{\|\vec{M}_{w}\|\|(\vec{u}_0)\|}\]
is positively bounded.
\end{cor}
\begin{proof}[Proof of Proposition \ref{Hölder exponent general formula}]
Now let $s \in [0, 1]$, and let $s' \in [0, 1]$ be a point close to, but not equal to, $s$. Let $n = \left\lfloor\frac{\ln|s'-s|}{\ln\left(\frac{1}{2}\right)}\right\rfloor$, so that $\frac{1}{2} \leq 2^n|s'-s| \leq 1$. Let $s = 0.a_1 \ldots a_n \ldots$ and $s' = 0.a'_1 \ldots a'_n \ldots$ be the respective binary expansions. Let us write $w = a_1 \ldots a_n$ and $w' = a'_1 \ldots a'_n$; we then have $s = h_{w}(S)$ and $s' = h_{w'}(S')$, for some $S, S' \in [0, 1]$, hence:
\begin{align}
\label{eq:prop91}
\frac{\ln \big\|u(s') - u(s)\big\|}{\ln \big|s'-s\big|}
&= \frac{\ln \big\|u\big(h_{w'}(S')\big) - u\big(h_{w}(S)\big)\big\|}{n \ln\left(\frac{1}{2}\right) + O(1)} \nonumber \\
&= \frac{\ln \big\|M_{w'}\big(u(S')\big) - M_{w}\big(u(S)\big)\big\|}{n \ln\left(\frac{1}{2}\right) + O(1)}.
\end{align}
Now two cases are possible:

\begin{itemize}
\begin{subequations}

\item Suppose $w' = w$. Then by the previous Corollary, we have
\begin{align}
\label{eq:prop92}
\ln \big\|M_{w'}\big(u(S')\big) - M_{w}\big(u(S)\big)\big\|
&= \ln \big\|\vec{M}_{w}\big(u(S') - u(S)\big)\big\| \nonumber \\
&= \ln \|\vec{M}_{w}\|\; +\; \ln \|u(S') - u(S)\|\; +\; O(1).
\end{align}

\item In the other case, by definition of $n$, $w$ and $w'$ are necessarily binary representations of two consecutive numbers. Without loss of generality, we may suppose that $s < s'$. Then we have $h_{w'}(0) = h_{w}(1)$, hence
\[\ln \big\|M_{w'}\big(u(S')\big)\; -\; M_{w}\big(u(S)\big)\big\|\;
=\; \ln \big\|\vec{M}_{w'}\big(u(S') - u(0)\big)\; +\; \vec{M}_{w}\big(u(1) - u(S)\big)\big\|.\]
Now it is easy to see that, given two vectors $\vec{u}, \vec{v}$ that can take values in a salient cone, the ratio between $\|\vec{u} + \vec{v}\|$ and $\|\vec{u}\| + \|\vec{v}\|$ is positively bounded. Since both terms in the above sum lie in $\vec{K}$, we can write, using once again the previous Corollary:
\begin{align*}
&\ln \big\|M_{w'}\big(u(S')\big) - M_{w}\big(u(S)\big)\big\|\; = \\
&\;=\; \ln \Big(\big\|\vec{M}_{w'}\big(u(S') - u(0)\big)\big\|\;
            +\; \big\|\vec{M}_{w}\big(u(1) - u(S)\big)\big\|\Big)\; +\; O(1)\\
&\;=\; \max\Big(\ln \big\|\vec{M}_{w'}\big(u(S') - u(0)\big)\big\|,\quad
                  \ln \big\|\vec{M}_{w}\big(u(1) - u(S)\big)\big\|\Big)\; +\; O(1)\\
&\;=\; \max\Big(\ln \big\|\vec{M}_{w'}\big\| + \ln \big\|u(S') - u(0)\big\|,\quad
                  \ln \big\|\vec{M}_{w}\big\| + \ln \big\|u(1) - u(S)\big\|\Big)\; +\; O(1).
\end{align*}
Now note that we have
\[\vec{M}_{w}(\vec{v}_0) = \vec{M}_{w'}(\vec{v}_1).\]
Indeed, we may write $w = a_0 \ldots a_k 01 \ldots 1$ and $w' = a_0 \ldots a_k 10 \ldots 0$ (where both last groups of bits have equal, either zero or nonzero, length); the identity then follows from the equalities $\vec{M}_0(\vec{v}_0) = \frac{3}{5}\vec{v}_0$, $\vec{M}_1(\vec{v}_1) = \frac{3}{5}\vec{v}_1$ and $\vec{M}_0(\vec{v}_1) = \vec{M}_1(\vec{v}_0)$, that are straightforward to check (see also Figure \ref{fig:vech}). Since $\vec{v}_{0, 1} \in \vec{K}$, we get
\[\ln \|\vec{M}_{w}\| + \ln \|\vec{v}_0\| + O(1)\; =\; \ln \|\vec{M}_{w'}\| + \ln\|\vec{v}_1\| + O(1),\]
or in other words, since $\|\vec{v}_0\| = \|\vec{v}_1\|$,
\[\ln \|\vec{M}_{w'}\|\; =\; \ln \|\vec{M}_{w}\| + O(1).\]
This allows us to write
\begin{align}
\label{eq:prop93}
&\ln \big\|M_{w'}\big(u(S')\big) - M_{w}\big(u(S)\big)\big\|\; = \\
&\; =\; \ln \|\vec{M}_{w}\|\; +\; \max\Big(\ln \big\|u(S') - u(0)\big\|,\; \ln \big\|u(1) - u(S)\big\|\Big)\; +\; O(1). \nonumber
\end{align}

\end{subequations}
\end{itemize}

Let us now check that in both cases, the middle term is bounded. Since $[0, 1]$ is compact and $u$ is continuous, it is clearly bounded from above. For the other inequality, note that, in the first case, we have $(S' - S)\, =\, 2^n(s' - s)\, \geq\, \frac{1}{2}\, >\, \frac{1}{4}$; and in the second case, we have $(S' - 0)\, +\, (1 - S)\, =\, 2^n(s' - s)\, \geq\, \frac{1}{2}$, hence either $(S' - 0) \geq \frac{1}{4}$ or $(1 - S) \geq \frac{1}{4}$. Now consider the set
\[X := \setsuch{(s_1, s_2) \in [0, 1]^2}{|s_1 - s_2| \geq \frac{1}{4}}\]
and the function
\begin{align*}
X          &\to     \mathbb{R}_{\geq 0},\\
(s_1, s_2) &\mapsto \|u(s_1) - u(s_2)\|.
\end{align*}
Since it is continuous and $X$ is compact, it reaches its minimum value. But by Corollary \ref{u injective}, it cannot vanish; hence it is positively bounded. Thus $\|u(S') - u(S)\|$ in the first case, and $\max\big(\|u(S') - u(0)\|,\; \|u(1) - u(S)\|\big)$ in the second case, are positively bounded.

This allows us to merge \eqref{eq:prop92} with \eqref{eq:prop93}, and to write:
\[\ln \big\|M_{w'}\big(u(S')\big) - M_{w}\big(u(S)\big)\big\|\; =\; \ln \|\vec{M}_{w}\| + O(1),\]
and if we plug this into \eqref{eq:prop91}, the result follows.
\end{proof}

\subsection{Case of rational points}
\label{sec:rational}

\begin{thm}
\label{Hölder exponent formula}
Let $s \in [0, 1] \cap \mathbb{Q}$. Then we know that the binary expansion of $s = 0,a_1a_2\ldots$ is eventually periodic; let $n$ be the length of its period, and $p = a_k \ldots a_{k+n-1}$ its period. Then $\vec{M}_{p}$ has distinct, positive, real eigenvalues, and we have
\[\alpha(u, s) = \frac{\ln \lambda_{p}}{n \ln \frac{1}{2}},\]
where $\lambda_{p}$ is the bigger eignevalue of $\vec{M}_{p}$.
\end{thm}
In practice, since we know the determinant of $\vec{M}_p$, we can calculate $\lambda_p$ by the formula
\[\lambda_{p} = \frac{1}{2}\left(\tr \vec{M}_{p} + \sqrt{ (\tr \vec{M}_{p})^2 - \left(\frac{3}{25}\right)^n }\right).\]

Note that Proposition \ref{Hölder exponent and derivative} also allows us to infer $u'(s)$. Indeed, we shall see in the next section that we never have $\alpha(u, s) = 1$.

To prove the theorem, we shall need the following lemma:
\begin{lem}
\label{Reduction to periodic case}
Let $s \in [0, 1]$, let $w = a_1 \ldots a_n$ be a word on the alphabet $\{0, 1\}$. Then we have:
\[\alpha_{\inf}\big(u, h_w(s)\big) = \alpha_{\inf}\big(u, s\big);\]
\[\alpha_{\sup}\big(u, h_w(s)\big) = \alpha_{\sup}\big(u, s\big).\]
\end{lem}
\begin{proof}
Let $0.b_1b_2\ldots$ be the binary expansion of $s$; let $\vec{u}_0 \in \vec{K}$. It is well-known that $\frac{\|\vec{M}_{w}(\vec{u})\|}{\|\vec{u}\|}$ is bounded when $\vec{u}$ varies; since $\vec{M}_{w}$ is invertible, the reciprocal of that quantity is bounded as well; hence the quantity
\[\ln \big\| \vec{M}_{a_1 \ldots a_kb_1 \ldots b_n}(\vec{u}_0) \big\|\;
-\; \ln \big\| \vec{M}_{b_1 \ldots b_n}(\vec{u}_0) \big\|\]
is bounded. The result follows immediately from Proposition \ref{Hölder exponent general formula}.
\end{proof}
\begin{proof}[Proof of Theorem \ref{Hölder exponent formula}]
The lemma we have just shown allows us to drop the preperiod of the expansion of $s$: we may suppose that it is periodic. From the formula \eqref{eq:Hölderexpformula} it follows that we have:
\[\alpha_{\inf}(u, s) = \alpha_{\sup}(u, s)
= \lim_{k \to \infty}
   \frac 
   	{\ln \| \vec{M}_{p}^k \|}
   	{kn \ln \left(\frac{1}{2}\right)}\]
Now since $\vec{K}$ is stable by $\vec{M}_p$, we can also say that $\tilde{K}$ is stable by $\tilde{M}_p$. Since $\tilde{M}_p$ is contractive, this means that it has exactly one fixed point in $\tilde{K}$, which means that $\vec{M}_p$ has a unique eigenvector $\vec{v}_p$ lying in $\vec{K}$ and associated with a real positive eigenvalue $\lambda_p$ (we will justify in a minute that it is indeed the bigger one). Let $\mu_p$ be the other eigenvalue; since $\det(\vec{M}_p) = (\frac{3}{25})^n > 0$, it is also real positive, and thus the associated eigenvector $\vec{w}_p$ cannot lie in $\vec{K}$. Using once again the fact that $\vec{M}_p(\vec{K})$ is a proper subset of $\vec{K}$ and considering the positions of the relevant vectors relative to each other and to $\vec{K}$, it can easily be seen that $\lambda_p > \mu_p$. It follows that $\| \vec{M}_{p}^k \| \sim \lambda_p^k$, hence the result.
\end{proof}

\begin{exm}
\label{Dyadic rational}
It follows from Theorem \ref{Hölder exponent formula} that whenever $s$ is dyadic rational, we have $\alpha(u, s) = \frac{\ln \frac{3}{5}}{\ln \frac{1}{2}} \approx 0,737$ (indeed, in this case, $\lambda_p$ is simply the bigger eigenvalue of $\vec{M}_0$ or $\vec{M}_1$, namely $\frac{3}{5}$). In particular, we have in this case $u'(s) = \infty$.
\end{exm}

Note that Lemma \ref{Reduction to periodic case} is interesting in itself, even outside the case $s \in \mathbb{Q}$: it shows that changing a finite number of bits in the binary expansion of $s$ does not change its Hölder exponent. For completeness' sake, let us also mention the following (fairly obvious) symmetry of $u$. Let
\[P =
\begin{pmatrix}
0 & 1 & 0\\
1 & 0 & 0\\
0 & 0 & 1\\
\end{pmatrix}\]
(in Figure \ref{fig:ut}, it simply corresponds to the vertical symmetry axis). Then
\begin{prop}
For all $s \in [0, 1]$, we have $u(1-s) = Pu(s)$.
\end{prop}
\begin{cor}
For all $s \in [0, 1]$, $\alpha_{\inf}(u, 1-s) = \alpha_{\inf}(u, s)$ and $\alpha_{\sup}(u, 1-s) = \alpha_{\sup}(u, s)$.
\end{cor}
\begin{proof}
Clearly, $P^2 = 1$, $M_1 = PM_0P^{-1} = PM_0P$, $M_0 = PM_1P$, and $K$ is stable by $P$. Let $s = 0.a_1a_2\ldots$ be its binary expansion; then, simply using \eqref{eq:formulau}, we have:
\begin{align*}
u(1-s) &= u(0.\bar{a}_1\bar{a}_2 \ldots)\\
       &= \lim_{n \to \infty}M_{\bar{a}_1 \ldots \bar{a}_n}u_0\\
       &= \lim_{n \to \infty}PM_{a_1 \ldots a_n}Pu_0\\
       &= Pu(s),\\
\end{align*}
where $\bar{a}$ is shorthand for $1 - a$ and $u_0 \in K$ is an arbitrary starting point.
\end{proof}

\subsection{Derivative of $u$}
\label{sec:derivative}

\begin{thm}
\label{no finite derivative}
For all $s \in [0, 1]$, $u'(s)$ is either $0$, $\infty$ or not defined at all.
\end{thm}
\begin{proof}
Assume the opposite: let $s$ be a number in $[0, 1]$ such that $u'(s)$ exists, is finite and positive. Then we have, in particular:
\[\lim_{t \to s} \frac{\|u(t) - u(s)\|}{|t - s|} = C\]
for some finite positive constant $C$.

Let $s = 0.a_1a_2 \ldots$ be a binary expansion. Let $t_n$ be a sequence, taking values in $[0, 1]$, such that for all $n \in \mathbb{N}$, the binary expansion $t_n$ coincides with that of $s$ in the first $n$ places. Obviously, such a sequence converges to $s$. We then have, for all $n \in \mathbb{N}$:
\[\frac{\big\|u(t_n) - u(s)\big\|}{|t_n - s|}\;
=\; \frac{\big\|\vec{M}_{a_1 \ldots a_n}\big(u(T_n) - u(S_n)\big)\big\|}{2^{-n}|T_n - S_n|},\]
where $S_n = h_{a_1 \ldots a_n}^{-1}(s) = 0.a_{n+1}a_{n+2} \ldots$ and $T_n = h_{a_1 \ldots a_n}^{-1}(t_n)$. Since $[0, 1]$ is compact, we may extract from the natural numbers a subsequence $\phi(n)$ such that $S_{\phi(n)}$ converges to some value $S_{\infty}$.

Suppose that $S_{\infty} < 1$ (by reversing the order of $s$ and $t$, we can similarly treat the case $S_{\infty} > 0$). From now on, we shall require that $t_n > s$. By Corollary \ref{Matrix-vector separation}, we know that
\[\frac{\big\|\vec{M}_{a_1 \ldots a_n}\big\|\big\|\big(u(T_n) - u(S_n)\big)\big\|}
       {\big\|\vec{M}_{a_1 \ldots a_n}\big(u(T_n) - u(S_n)\big)\big\|}\]
is positively bounded. We also know that
\[\frac{\big\|\vec{M}_{a_1 \ldots a_n}\big(u(T_n) - u(S_n)\big)\big\|}{2^{-n}|T_n - S_n|}\]
tends to $C$, hence it is positively bounded as well. Suppose for a moment that $t_n$ is such that $T_n$ is identically equal to $1$. Then for sufficiently large $n$, $S_{\phi(n)}$ varies in a closed interval that does not contain $1$, and since $u$ is continuous and injective, $\frac{\|(u(T_{\phi(n)}) - u(S_{\phi(n)}))\|}{|T_{\phi(n)} - S_{\phi(n)}|}$ is positively bounded.

This means that $\frac{\vec{M}_{a_1 \ldots a_{\phi(n)}}}{2^{-\phi(n)}}$ is positively bounded in norm, so that it takes values in a compact space. Hence from $\phi(n)$ we can extract yet another subsequence, that we shall call $\psi(n)$, such that $\frac{\vec{M}_{a_1 \ldots a_{\psi(n)}}}{2^{-\psi(n)}}$ converges; let us call $\vec{M}_{\infty}$ its limit. Note that, though we used a particular value of $t_n$ to prove that this construction was possible, the actual objects that we have constructed do not depend on it.

Now let $T$ be any point in $(S_{\infty}, 1]$; we set $t_n = h_{a_1 \ldots a_n}(T)$, so that for all $n$, $T_n = T$. We then have:
\begin{align*}
\lim_{n \to \infty} \frac{\big\|u(t_{\psi(n)}) - u(s)\big\|}{|t_{\psi(n)} - s|}\;
&=\; \lim_{n \to \infty} \frac{\big\|\vec{M}_{a_1 \ldots a_{\psi(n)}}\big(u(T) - u(S_{\psi(n)})\big)\big\|}
                            {2^{-{\psi(n)}}|T - S_{\psi(n)}|}\\
&=\; \frac{\big\|\vec{M}_{\infty}\big(u(T) - u(S_{\infty})\big)\big\|}{|T - S_{\infty}|},\\
\end{align*}
which means that the latter ratio is equal to $C$ regardless of the value of $T$.

This claim is definitely too strong to be true, and we shall soon prove it is absurd. Note that we have
\[\det\left(\frac{\vec{M}_{a_1 \ldots a_{\psi(n)}}}{2^{-\psi(n)}}\right) = \det\left(2\vec{M}_i\right)^{\psi(n)} = \left(\tfrac{12}{25}\right)^{\psi(n)},\]
which tends to $0$ as $n$ increases. Hence $\det\left(\vec{M}_{\infty}\right) = 0$, in other words, $\vec{M}_{\infty}$ has rank $1$. The claim can then be reformulated as follows: there exists some linear form $\Phi$ such that the ratio
\[\frac{\Phi\big(u(T) - u(S_{\infty})\big)}{T - S_{\infty}}\]
is constant. The interval $(S_{\infty}, T)$ contains infinitely many dyadic rational points; by Corollary \ref{properties of Du}, we may find in this interval a dyadic rational value $T'$ such that $\Phi\big(\vec{D}u(T')\big) \neq 0$. But Example \ref{Dyadic rational} together with Lemma \ref{Link between u and uacb} tell us that in this case, $\Phi\big(u(T) - u(S_{\infty})\big)$ as a function of $T$ has an infinite derivative in $T'$. So must the whole ratio: contradiction.

Now let $s \in [0, 1] \cap \mathbb{Q}$, and suppose $u'(s)$ is undefined. Let us apply Theorem \ref{Hölder exponent formula}: it is only possible if $\alpha(u, s) = 1$. Let $p$ be the period of the binary expansion of $s$, $n$ the length of the period; then we have $\lambda_{p} = (\frac{1}{2})^n$. We know that $M_{p}$ has determinant $(\frac{3}{25})^n$ and that one of its eigenvalues is equal to $1$. Hence $\tr M_{p} = 1 + (\frac{1}{2})^n + (\frac{6}{25})^n$, and $\tr 25^nM_{p} = 25^n + 6^n + (\frac{25}{2})^n$. On the other hand, we know that in the canonical basis, $25M_0$ and $25M_1$ are given by matrices with integer coefficients, which means that so is $25^nM_{p}$, hence $\tr 25^nM_{p}$ must be integer. Contradiction.
\end{proof}

\begin{cor}
For Lebesgue-almost all $s \in [0, 1]$, $u'(s) = 0$.
\end{cor}
\begin{proof}
Let $\Phi$ be a linear form whose kernel does not meet $\vec{K}$; then from Corollary \ref{properties of Du}, it is easy to deduce that the function $\Phi \circ u$ is monotonic. But it is a well-known result (see for example \cite{rudin}, Theorem 8.19) that a monotonic function is differentiable (in its usual sense of ``has a finite derivative'') almost everywhere. Since by Lemma \ref{Link between u and uacb}, $\Phi \circ u$ behaves in the same way as $u$ itself, the result follows.
\end{proof}

\subsection{Numerical values and estimations of Hölder exponent}
\label{sec:numerical}

Here is a table of all possible periods shorter than or equal to 7 (up to cyclic permutation), ordered by Hölder exponent:

\vspace{10 pt}
\begin{tabular}{|c|c|c|c|c|}
\hline
$s$ & Period $p$ & Length $n$ & $5^n \tr \vec{M}_{p}$ & Approximate value of $\alpha$ \\ \hline
$\frac{1}{3}$     & 01      & 2 & 7    & 1.119 \\ \hline
$\frac{21}{127}$  & 0010101 & 7 & 388  & 1.096 \\ \hline
$\frac{11}{63}$   & 001011  & 6 & 175  & 1.086 \\ \hline
$\frac{5}{31}$    & 00101   & 5 & 76   & 1.085 \\ \hline
$\frac{1}{5}$     & 0011    & 4 & 34   & 1.078 \\ \hline
$\frac{19}{127}$  & 0010011 & 7 & 436  & 1.072 \\ \hline
$\frac{11}{127}$  & 0001011 & 7 & 472  & 1.055 \\
$\frac{13}{127}$  & 0001101 & 7 & 472  & 1.055 \\ \hline
$\frac{1}{7}$     & 001     & 3 & 16   & 1.050 \\ \hline
$\frac{3}{31}$    & 00011   & 5 & 88   & 1.040 \\ \hline
$\frac{5}{63}$    & 000101  & 6 & 211  & 1.039 \\ \hline
$\frac{1}{9}$     & 000111  & 6 & 223  & 1.025 \\ \hline
$\frac{9}{127}$   & 0001001 & 7 & 580  & 1.012 \\ \hline \hline
$\frac{5}{127}$   & 0000101 & 7 & 616  & 0.999 \\ \hline
$\frac{1}{21}$    & 000011  & 6 & 250  & 0.997 \\ \hline
$\frac{7}{127}$   & 0000111 & 7 & 628  & 0.995 \\ \hline
$\frac{1}{15}$    & 0001    & 4 & 43   & 0.982 \\ \hline
$\frac{3}{127}$   & 0000011 & 7 & 736  & 0.962 \\ \hline
$\frac{1}{31}$    & 00001   & 5 & 124  & 0.936 \\ \hline
$\frac{1}{63}$    & 000001  & 6 & 367  & 0.903 \\ \hline
$\frac{1}{127}$   & 0000001 & 7 & 1096 & 0.880 \\ \hline
0                 & 0       & 1 & 4    & 0.737 \\ \hline
\end{tabular}
\vspace{10 pt}

When we look at this table, it becomes apparent that globally, numbers whose binary expansions contain long clusters of equal consecutive bits tend to have smaller Hölder exponents, while those whose binary expansions are more ``mixed'' tend to have bigger Hölder exponents. The exponent is lowest when the binary expansion ends in an unbroken string of zeroes or ones (that is, for dyadic rational numbers --- see below), and seems highest when zeroes and ones strictly alternate (that is, for $\frac{1}{3}$ and its dyadic-rational multiples; the author does not know how to prove this). The author was able to prove the following statement, that partly confirms this intuition:
\begin{prop}
\label{bound on exponent}
Let $s \in [0, 1]$, $s = 0.a_1a_2 \ldots$ its binary expansion. Let us define
\[d_{\inf,\, \sup}(s)\; :=\; \operatorname*{lim\,inf,\,sup}_{n \to \infty} \frac{1}{n} \sum_{i=1}^{n} |a_i - a_{i+1}|.\]
We then have:
\begin{align*}
\alpha_{\inf}(u, s)\;
&\leq\; \frac{\ln \frac{3}{5}}{\ln \frac{1}{2}} + d_{\inf}(s) ;\\
\alpha_{\sup}(u, s)\;
&\leq\; \frac{\ln \frac{3}{5}}{\ln \frac{1}{2}} + d_{\sup}(s).
\end{align*}
\end{prop}
Informally, $d_{\inf}(s)$ and $d_{\sup}(s)$ give bounds for the asymptotic proportion of places where the bit changes value; the latter can also be understood as the reciprocal of the asymptotic average length of blocs of consecutive equal bits.
\begin{cor}
If $d_{\sup}(s) < 1 - \frac{\ln \frac{3}{5}}{\ln \frac{1}{2}} \approx 0.263 \approx \frac{1}{3.802}$, then $u'(s) = \infty$.
\end{cor}
To prove the proposition, we will need a small technical lemma:
\begin{lem}
\label{Technical lemma for upper bound}
For any $\vec{u} \in \vec{K}$, $n \in \mathbb{N}$ and $i \in \{0, 1\}$, we have
\[\frac{\big\| \vec{M}_i^n(\vec{u}) \big\|_2}{\big\| \vec{u} \big\|_2}\; \geq\; \frac{1}{2}\left(\frac{3}{5}\right)^n.\]
\end{lem}
\begin{proof}
Let us first express $\vec{K}$ in the bases that diagonalise $\vec{M}_i$. A quick calculation shows that we have:
\[\vec{K}\; =\; \setsuch{ a\vec{v}_0 + b\vec{w}_0 }{ 0 \leq \frac{b}{a} \leq 3 }\; =\; \setsuch{ a\vec{v}_1 + b\vec{w}_1 }{ 0 \geq \frac{b}{a} \geq -3 }\]
Let $i \in \{0, 1\}$ and $\vec{u} = a\vec{v}_i + b\vec{w}_i \in \vec{K}$. We then have $\vec{M}_i(\vec{u}) = \frac{3}{5}a\vec{v}_i + \frac{1}{5}b\vec{w}_i$. Now let $n \in \mathbb{N}$; we have (since $\vec{v}_i$ and $\vec{w}_i$ are orthogonal):
\begin{align*}
\left(\frac{\big\| \vec{M}_i^n(\vec{u}) \big\|_2}
           {\big\| \vec{u} \big\|_2}\right)^2\;
  &=\; \frac{\big\| \left(\frac{3}{5}\right)^n a\vec{v}_i + \left(\frac{1}{5}\right)^n b\vec{w}_i \big\|_2^2}
            {\big\| a\vec{v}_i + b\vec{w}_i \big\|_2^2} \\
  &=\; \frac{\left(\frac{3}{5}\right)^{2n} \big\| a\vec{v}_i \big\|_2^2 +
             \left(\frac{1}{5}\right)^{2n} \big\| b\vec{w}_i \big\|_2^2}
            {\big\| a\vec{v}_i \big\|_2^2 + \big\| b\vec{w}_i \big\|_2^2} \\
  &\geq\; \left(\tfrac{3}{5}\right)^{2n} \frac{\big\| a\vec{v}_i \big\|_2^2}
                                              {\big\| a\vec{v}_i \big\|_2^2 +
                                               \big\| b\vec{w}_i \big\|_2^2} \\
  &=\; \left(\tfrac{3}{5}\right)^{2n} \frac{1}{1 + \left(\frac{b}{a}\right)^2 \frac{\| \vec{w}_i \|_2^2}{\| \vec{v}_i \|_2^2}} \\
  &=\; \left(\tfrac{3}{5}\right)^{2n} \frac{1}{1 + \frac{1}{3}\left(\frac{b}{a}\right)^2} \\
  &\geq\; \left(\tfrac{3}{5}\right)^{2n} \frac{1}{1 + \frac{1}{3}3^2} \\
  &=\; \left(\tfrac{1}{2}(\tfrac{3}{5})^n\right)^2. \\
\end{align*}
\end{proof}
\begin{proof}[Proof of Proposition \ref{bound on exponent}]
Let $\vec{u}_0 \in \vec{K}$; by Proposition \ref{Hölder exponent general formula} we may then write
\[\alpha_{\inf}\left(u, s\right)\;
=\; \liminf_{n \to \infty}
   \frac 
   	{\ln \big\| \vec{M}_{a_1 \ldots a_n} \big\|}
   	{n \ln \left(\frac{1}{2}\right)}\;
=\; \liminf_{n \to \infty}
   \frac 
   	{\ln \big\| \vec{M}_{a_1 \ldots a_n}(\vec{u}_0) \big\|_2}
   	{n \ln \left(\frac{1}{2}\right)}.\]
We may now decompose the binary expansion of $s$ into clusters of consecutive equal bits. Let $n \in \mathbb{N}$; we have:
\[\vec{M}_{a_1 \ldots a_n}\; =\; \vec{M}_{a_1}\ldots\vec{M}_{a_n}\; =\; \vec{M}_{\alpha_1}^{l_1}\vec{M}_{\alpha_2}^{l_2}\ldots\vec{M}_{\alpha_{k(n)-1}}^{l_{k(n)-1}}\vec{M}_{\alpha_{k(n)}}^{r(n)},\]
where the values of $\alpha_i$ alternate between $0$ and $1$ (so that $\alpha_1 = a_1,\; \alpha_2 = 1-a_1,\; \alpha_3 = a_1,\; \ldots,\; \alpha_{k(n)} = a_n$) and all $l_i$ are positive. From Lemma \ref{Technical lemma for upper bound}, it follows that
\[\big\| \vec{M}_{a_1 \ldots a_n}(\vec{u}_0) \big\|_2\;
\geq\; \left(\tfrac{1}{2}\right)^{k(n)}\left(\tfrac{3}{5}\right)^{n}\big\|\vec{u}_0\big\|_2\]
hence (don't forget that $\ln \left(\frac{1}{2}\right) < 0$):
\[\frac{\ln \big\| \vec{M}_{a_1 \ldots a_n}(\vec{u}_0) \big\|_2}{n \ln \left(\frac{1}{2}\right)}\;
\leq\;   \frac{k(n)}{n}
       + \frac{\ln \left(\frac{3}{5}\right)}{\ln \left(\frac{1}{2}\right)}
       + \frac{\ln \big\|\vec{u}_0\big\|_2}{n}.\]
When $n$ grows, the third term tends to $0$, and the first term's lower limit is by definition equal to $d_{\inf}(x)$, hence the conclusion.
\end{proof}

Finding lower bounds for the Hölder exponent seems much harder than finding upper bounds. Here is a trivial one: we obviously have
\[\forall \vec{u} \in \vec{\mathcal{H}},\; \forall i \in \{0, 1\},\quad
\big\|M_i(\vec{u})\big\|\; \leq\; \tfrac{3}{5}\big\|\vec{u}\big\|,\]
hence the Hölder exponent is always greater or equal to $\frac{\ln \left(\frac{3}{5}\right)}{\ln \left(\frac{1}{2}\right)} \approx 0.737$. However, this is useless for finding points where the derivative is zero. The author suspects that if the binary expansion of $s$ never contains more than two equal consecutive bits, then $u'(s) = 0$ --- but it has yet to be checked.

\end{document}